\documentclass[11pt]{amsart}
\usepackage{xcolor}
\usepackage{graphicx}
\usepackage{latexsym}
\usepackage{amsfonts,amsmath,amssymb}
\newtheorem{theorem}{Theorem}[section]
\newtheorem{proposition}[theorem]{Proposition}

\newtheorem{corollary}[theorem]{Corollary}

\usepackage{color}

\def\eps{\varepsilon} 
\def\la{\lambda}
\def\a{\alpha}

\def\part{\partial}

\def\b1{\bold 1}

\newcommand{\beq}{\begin{equation}}
\newcommand{\eeq}{\end{equation}}

\theoremstyle{remark}

\numberwithin{equation}{section}
\date{\today}

\begin{document}
\title[Counting records]{Counting records in a random, non-uniform, permutation}
\author{Boris Pittel}
\address{Department of Mathematics, Ohio State University, 231 West 18-th Avenue, Columbus OH 43210-1175}
\email{pittel.1@osu.edu}
\begin{abstract} Counting permutations of $[n]$ by the number of records, i.e. left-to-right maxima, is a classic problem in combinatorial enumeration. In the first volume of ``The Art of Computer Programming", Donald Knuth demonstrated its relevance for analysis of average case complexity of a basic algorithm for determining a maximum in a linear list of numbers. It is well known that the expected, and likely, number of those records in a {\it uniformly\/} random permutation is asymptotic to $\log n$. Cyril Banderier, Rene Beier, and Kurt Mehlhorn studied the case of a non-uniform random permutation, which is obtained from a generic permutation of $[n]$ by selecting its elements one after another independently with probability $p$, and permuting the selected elements uniformly at random. They proved that $E_n(p)$, the largest expected number of the maxima, is between $\text{const}\sqrt{n/p}$ and $O\bigl(\sqrt{(n/p)\log n}\bigr)$ if $p$ is fixed. For $p\gg 1/n$ and simultaneously $1-p\ge \text{const }n^{-1/2}\log n$, we prove that $E_n(p)$ is exactly of order $(1-p)\sqrt{n/p}$. 
\end{abstract}
\maketitle
 \section{Introduction} Typically, the average case analysis of combinatorial algorithms is predicated on an assumption that the random instance of a problem is chosen uniformly at random from a given set of those instances. This assumption opens the door to powerful combinatorial methods, but inevitably limits applicability of the results to problem with considerably lesser degree of homogeneity and symmetry. In a pioneering paper \cite{BMB}, Cyril Banderier, Rene Beier, and Kurt Mehlhorn undertook a study of three algorithmic problems, in which randomization is confined to a {\it given\/} combinatorial structure. This allows, in principle, to consider a hybrid version of  algorithmic analysis, with focus on the worst {\it average}-case behavior of an algorithm applied to the random problem instance chosen from a priori restricted set of instances. The numerous references in the paper allow the authors to put their work in a proper perspective. The interested reader may also wish to consult Volume 1, and new Volume 4B of the series ``The Art of Computer Programming'' by Donald Knuth \cite{Knu} for a detailed discussion of "Find the maximum" algorithm and a succinct description of the related results from 
\cite{BMB}. 

Here is one of these  results. Let $\bar X$ be an arbitrary permutation of $[n]$. Let $X=X(\bar X)$ denote the permutation obtained by selecting the elements of $\bar X$ independently, each with probability $p$, and permuting uniformly at random the selected elements. Let $\la(X)$ denote the total number of records, i. e. left-to-right maxima, in $X$. It was proved in \cite{BMB} that, for a fixed $p$, the expected value $\Bbb E[\la(X)]=O\bigl(\sqrt{(n/p)\log n}\bigr)$, uniformly over all $\bar X$, and that $\max_{\bar X}\Bbb E[\la(X)]$ is of order $\sqrt{n/p}$ at least.

Our main result is that, in fact, for $p\gg 1/n$ and simultaneously $1-p\ge \text{const } n^{-1/2}\log n$,  $\max_{\bar X}\Bbb E[\la(X)]$ is of order $(1-p)\sqrt{n/p}$ exactly. While our methods are quite different, it is the study in \cite{BMB} that made our research possible.

Kurt Mehlhorn \cite{Meh} drew my attention to a recent paper \cite{Bei} by Beier, R\H oglin, R\H osner, and V\H ocking, 
where the model of smoothed analysis with problem instances subject to a random noise is applied to a class of integer optimization problems.
\section{Statements and proofs.}

Let $\bar X$ be an arbitrary permutation of $[n]$. Let $X$ denote the permutation obtained by marking (selecting) the elements of $\bar X$ independently, each with probability $p$, and permuting uniformly at random the marked elements. Let $\la(X)$ denote the total number of left-to-right maxima in $X$. We write $\la(X)=\la_{\text{u}}(X)+\la_{\text{m}}(X)$ where $\la_{\text{u}}(X)$ and $\la_{\text{m}}(X)$ are the numbers of the left-to-right maxima of $X$ that are, respectively, unmarked and marked elements of $\bar X$.
\begin{theorem}\label{thm1}   
$\max_{\bar X}\Bbb E[\la_{\text{u}}(X)]=O\Bigl((1-p)\sqrt{\tfrac{n}{p}}\Bigr)$, if $p\gg n^{-1}$.
\end{theorem}

\begin{proof}

\noindent Let $A_k$, $(k\in [n])$, be the event that is $x_k$ is an unmarked left-to-right maximum in the resulting permutation $X$. We need to bound $\Bbb P\bigl(A_k)$. 
On the event $A_k$, we have
$\bar x_k=x_k\ge k$, since $x_k$ is the largest element among $x_1,\dots, x_k$.
Introduce the sets $N_{i,j}=N_{i,j}(\bar x_k)$:
\begin{equation*}
\begin{aligned}
&N_{1,1}=\{\bar x_{\a}:\a<k,\,  \bar x_{\a}<\bar x_k \}, \,\,\, N_{1,2}=\{\bar x_{\a}:  \a<k,\, \bar x_{\a}>\bar x_k\},\\
&N_{2,1}=\{\bar x_{\a}:\a>k,\,  \bar x_{\a}<\bar x_k\}, \,\,\, N_{2,2}=\{\bar x_{\a}:  \a>k,\, \bar x_{\a}>\bar x_k\};\\
\end{aligned}
\end{equation*}
$x_k$ cannot be a left-to-right maximum, unless $|N_{1,2}|+k\le n$. Let
$M_{i,j}$ stand for the set of marked elements of $N_{i,j}$. 
For the event in question, $\bar x_k$ is unmarked, so $M_{1,2}=N_{1,2}$; otherwise there is an unmarked element
 $\a\in N_{1,2}\setminus M_{1,2}$ exceeding $\bar x_k$ that will stay to the left of $\bar x_k$, and $\bar x_k$ will not be a left-to-right maximum in the resulting permutation $X$.  
Denote $n_{i,j}=|N_{i,j}|$, $m_{i,j}=|M_{i,j}|$; so $\sum_{i,j} n_{i,j}=n-1$, and 
$m_{1,2}=n_{1,2}$.

Let $k_1$ and $k_2$ denote generic numbers of elements from $M_{1,1}$ and from $M_{2,1}$ that will stay to the left of $\bar x_k$, and will be moved to the left of $\bar x_k$, respectively. For the event in question to hold, $k_1+k_2=k-1\le\bar x_k-1$. In addition, all those $n_{1,2}$ elements to the left of $\bar x_k$, exceeding $\bar x_k$, will need to be marked, and moved to the right of $\bar x_k$. Consequently, on the event in question $m:=\sum_{i,j}m_{i,j}\ge k-1+n_{1,2}$. The total number of such  permutations is
\begin{multline*}
(k-1)!\bigl(m-(k-1)\bigr)!\prod_{(i,j)\neq (1,2)}\binom{n_{i,j}}{m_{i,j}}\cdot \sum_{k_1+k_2=k-1}\binom{m_{1,1}}{k_1}\binom{m_{2,1}}{k_2}\\
=(k-1)!\bigl(m-(k-1)\bigr)!\prod_{(i,j)\neq (1,2)}\binom{n_{i,j}}{m_{i,j}}\cdot\binom{m_{1,1}+m_{2,1}}{k-1};
\end{multline*}
$(1,2)$ is excluded from the double product since $\binom{n_{1,2}}{m_{1,2}}=1$. To get the contribution of all such permutations with a given $m$ to $\Bbb P(A_k)$ we need to multiply the above product by 
$\tfrac{p^m(1-p)^{n-m}}{m!}$ and evaluate the sum of all these fractions over $\{m_{i,j}\}$ subject to
constraints $\sum_{i,j}m_{i,j}=m$, $m_{1,2}=n_{1,2}$. We will repeatedly use the identity
\begin{equation}\label{classic}
\sum_{a_1+a_2=a}\binom{b_1}{a_1}\binom{b_2}{a_2}=\binom{b_1+b_2}{a}.
\end{equation}
To begin, denoting $n_{\circ,1}=n_{1,1}+n_{2,1}$,
$m_{\circ,1}=m_{1,1}+m_{2,1}$, we have
\begin{multline*}
\sum_{\{m_{i,j}\}}\prod_{(i,j)\neq (1,2)}\binom{n_{i,j}}{m_{(i,j)}}\cdot\binom{m_{1,1}+m_{2,1}}{k-1}\\
=\frac{n_{\circ,1}!}{(n_{\circ,1}-(k-1))!}\sum_{m_{\circ,1},\, m_{2,2}}\binom{n_{\circ,1}-(k-1)}{m_{\circ,1}-(k-1}\binom{n_{2,2}}{m_{2,2}}\\
=(k-1)!\binom{n_{\circ,1}}{k-1}\binom{n-n_{1,2}-k}{m-n_{1,2}-(k-1)};
\end{multline*}
as $n_{\circ,1}+n_{2,2}=n-1-n_{1,2}$. Since $n_{\circ,1}=\bar x_k-1$,  we obtain  
\begin{multline}\label{new1}
\Bbb P(A_k)=\sum_{n>m\ge n_{1,2}+k-1} \binom{m}{k-1}^{-1}\\
\times \binom{\bar x_k-1}{k-1}\binom{n-n_{1,2}-k}{m-n_{1,2}-(k-1)}\cdot p^m(1-p)^{n-m}.
\end{multline}
%
By \eqref{new1}, and a classic formula 
\begin{equation}\label{0.9new}
 \frac{a_1!\,a_2!}{(a_1+a_2+1)!}=\int_0^1t^{a_1}(1-t)^{a_2}\,dt,
\end{equation}
 we have
\begin{multline*}
\Bbb P(A_k)
=(1-p)^n\binom{\bar x_k-1}{k-1}\\
\times\int_0^1\left(\frac{p\xi}{1-p}\right)^{k-1}\left(\frac{p(1-\xi)}{1-p}\right)^{n_{1,2}}\\
\times \sum_{n>m\ge n_{1,2}+k-1}\!\!\!\!\!\!(m+1)\,
\binom{n-(n_{1,2}+k)}{m-(n_{1,2}+k-1)}\left(\frac{p(1-\xi)}{1-p}\right)^{m-(n_{1,2}+k-1)}\,d\xi.
\end{multline*}
By $m+1=(n_{1,2}+k)+(m-(n_{1,2}+ k-1))$, the bottom sum equals
\begin{multline*}
(n_{1,2}+k)\left(1+\frac{p(1-\xi)}{1-p}\right)^{n-(n_{1,2}+k)}\\
+\bigl(n-(n_{1,2}+k)\bigr)\frac{p(1-\xi)}{1-p}
\left(1+\frac{p(1-\xi)}{1-p}\right)^{n-(n_{1,2}+k)-1}\\
=\left(\frac{1-p\xi}{1-p}\right)^{n-(n_{1,2}+k)-1}\left[(n_{1,2}+k)\left(\frac{1-p\xi}{1-p}-\frac{p(1-\xi)}{1-p}\right)
+n\frac{p(1-\xi)}{1-p}\right]\\
=\left(\frac{1-p\xi}{1-p}\right)^{n-(n_{1,2}+k)-1}\cdot\left[n_{1,2}+k+n\frac{p(1-\xi)}{1-p}\right].
\end{multline*}
For $n\ge n_{1,2}+k$,
\begin{multline}\label{0.86new}
\Bbb P(A_k)=(1-p)^n\binom{\bar x_k-1}{k-1}\left(\frac{p}{1-p}\right)^{n_{1,2}+k-1}\\
\times\int_0^1\xi^{k-1}(1-\xi)^{n_{1,2}}
\cdot \left(\frac{1-p\xi}{1-p}\right)^{n-(n_{1,2}+k)-1}\cdot\left[n_{1,2}+k+n\frac{p(1-\xi)}{1-p}\right]\,d\xi\\
=\binom{\bar x_k-1}{k-1}p^{n_{1,2}+k-1}(1-p)\cdot\int_0^1 \xi^{k-1}(1-\xi)^{n_{1,2}}
(1-p\xi)^{n-(n_{1,2}+k)-1}\\
\times \bigl[(n_{1,2}+k)(1-p\xi)+\bigl(n-(n_{1,2}+k)\bigr)p(1-\xi)\bigr]\,d\xi\\
\le n\binom{\bar x_k-1}{k-1}p^{n_{1,2}+k-1}(1-p)\cdot\int_0^1 \xi^{k-1}(1-\xi)^{n_{1,2}}
(1-p\xi)^{n-(n_{1,2}+k)}\,d\xi.\\
\end{multline}
Switching from $\xi$ to $\eta=p\xi$, we transform the bound \eqref{0.86new} into
\begin{multline}\label{0.69new}
n\binom{\bar x_k-1}{k-1}\frac{1-p}{p}\int_0^p \eta^{k-1}(p-\eta)^{n_{1,2}}(1-\eta)^{n-(n_{1,2}+k)}\,d\eta\\
\le n\binom{\bar x_k-1}{k-1}\frac{1-p}{p}\int_0^p \eta^{k-1}(1-\eta)^{n-k}\,d\eta,\\
\end{multline}
since $\tfrac{p-\eta}{1-\eta}<1$. (Dependence on $n_{1,2}$ is gone!)
The last integrand attains its absolute maximum at $\bar \eta=\tfrac{k-1}{n-1}\ge p$ if $k\ge pn +1-p$. 
Consider $k\ge k_1:=\lceil(1+\eps)(n-1)p\rceil$. Since the integrand is log-concave, the integral is at most
\begin{multline*}
p^{k-1}(1-p)^{n-k}\int_0^p\exp\left(\left(\frac{k-1}{p}-\frac{n-k}{1-p}\right)(y-p)\right)\,dy\\
\le p^{k-1}(1-p)^{n-k}\int_0^p\exp\left(\frac{\eps(n-1)}{1-p}(y-p)\right)\,dy\\
\le \frac{1-p}{(n-1)\eps}\,p^{k-1}(1-p)^{n-k}.
\end{multline*}
Consequently, for $k\ge k_1$, the bound \eqref{0.69new} is at most of order
\[
\frac{(1-p)^2}{\eps p}\cdot \binom{n-1}{k-1}p^{k-1}(1-p)^{(n-1)-(k-1)}.
\]
Now $\bigl\{\binom{n-1}{j}p^{j}(1-p)^{n-1-j}:\,j\in [n-1]\bigr\}$ is the binomial distribution, with parameters $p$ and $n-1$. 
By a classic tail bound for the binomial distribution,  
Janson, \L uczak, and Ruci\'nski \cite{JLR},
\[
\sum_{k\ge k_1}\binom{n-1}{k-1}p^{k-1}(1-p)^{n-k}\le 2\exp\bigl(-\tfrac{\eps^2}{3}pn\bigr),\quad \eps\le 1.
\]
So, the {\it total\/} contribution to the bound in \eqref{0.86new} from unmarked records at locations  $k \in [k_2,n]$, 
with $k_2:= \lceil(1-\eps)(n-1)p\rceil$, is at most of order
\begin{equation}\label{0.1}
\eps pn+ \tfrac{(1-p)^2}{\eps p}\exp\bigl(-\tfrac{\eps^2}{3}pn\bigr)\le 2(1-p)\sqrt{\tfrac{n}{p}}, \quad\text{if } \eps=\eps(n):=
\tfrac{1-p}{\sqrt{pn}}.
\end{equation}

Consider $cn^{1/2}\le k\le k_2$, $c=c(p)$ to be chosen. This time, $\bar \eta=\tfrac{k-1}{n-1}$, the absolute maximum point of the bottom integrand in \eqref{0.69new} is within $(0,p)$. Denoting
$H(\eta)=(k-1)\log\eta -(n-k)\log (1-\eta)$, we see that 
\[
H''(\eta)=-\frac{k-1}{\eta^2}-\frac{n-k}{(1-\eta)^2},\quad
H'''(\eta)=\frac{2(k-1)}{\eta^3}-\frac{2(n-k)}{(1-\eta)^3}.
\]
So, $H''(\bar\eta)=-\frac{(n-1)^3}{(k-1)(n-k)}$, and $H'''(\eta)\ge 0$ for $\eta\le \bar\eta$. So,
\[
H(\eta)\le H(\bar\eta)-\frac{(n-1)^3}{2(k-1)(n-k)}(\eta-\bar\eta)^2,\quad \eta\le\bar\eta,
\]
and, for $\eta>\bar\eta$, and some $\tilde \eta\in [\bar\eta,\eta]$,
\begin{multline*}
H(\eta)=H(\bar\eta)-\frac{(n-1)^3}{2(k-1)(n-k)}(\eta-\bar\eta)^2+\frac{1}{6}H'''(\tilde\eta)(\eta-\bar\eta)^3\\
\le H(\bar\eta)-\frac{(n-1)^3}{2(k-1)(n-k)}(\eta-\bar\eta)^2+\frac{k-1}{3\tilde\eta^2}(\eta-\bar\eta)^3\\
\le H(\bar\eta)-\frac{(n-1)^3}{2(k-1)(n-k)}(\eta-\bar\eta)^2+\frac{k-1}{3\bar\eta^2}(\eta-\bar\eta)^2\\
=H(\bar\eta)-\frac{(n-1)^3}{2(k-1)(n-k)}(\eta-\bar\eta)^2\frac{1+2\bar\eta}{3}\le 
H(\bar\eta)-\frac{(n-1)^3}{6(k-1)(n-k)}(\eta-\bar\eta)^2.
\end{multline*}
It follows directly that
\begin{multline*}
\int_0^p \eta^{k-1}(1-\eta)^{n-k}\,d\eta\le e^{H(\bar\eta)}\int_0^p\exp\left(-\frac{(n-1)^3}{6(k-1)(n-k)}
(\eta-\bar\eta)^2\right)\,d\eta\\
\le e^{H(\bar\eta)}\sqrt{\frac{2\pi}{\left(\tfrac{(n-1)^3}{3(k-1)(n-k)}\right)}}\\
=\frac{(6\pi)^{1/2}}{n-1}\frac{(k-1)^{k-1}(n-k)^{n-k}}{(n-1)^{n-1}} (k-1)^{1/2}(n-k)^{1/2}\\
=O\Bigl(n^{-1}\binom{n-1}{k-1}^{-1}\Bigr).
\end{multline*}
So, the bottom bound in in \eqref{0.86new} is at most of order 
\[
\frac{1-p}{p}\binom{\bar x_k-1}{k-1}\binom{n-1}{k-1}^{-1}.
\]
It remains to bound the sum of these terms over $k\in [cn^{1/2},k_2]$, and select a proper value 
of $c=c(n)$. Observe that for $x\in (k,n]$
\[
\binom{x-1}{k}\binom{n-1}{k}^{-1}\binom{x-1}{k-1}^{-1}\binom{n-1}{k-1}=\frac{x-k}{n-k}\le 1.
\]
So, for 
$k> k(n):= \lceil c n^{1/2}\rceil$, whence $\bar x_k> k(n)$, we have
\begin{multline*}
\binom{\bar x_k-1}{k-1}\binom{n-1}{k-1}^{-1}\le \binom{\bar x_k-1}{k(n)}\binom{n-1}{k(n)}^{-1}=\frac{(\bar x_k-1)_{k(n)}}{(n-1)_{k(n)}}\le \left(\frac{\bar x_k}{n}\right)^{k(n)}.\\
\end{multline*}
Therefore $\sum_{k\in [k(n), k_2]}\Bbb P(A_k)$ is at most of order 
\begin{multline*}
\frac{1-p}{p}\sum_{k\in [k(n), k_2]} \left(\frac{\bar x_k}{n}\right)^{k(n)}
\le \frac{1-p}{p}\sum_{k=1}^n \left(\frac{k}{n}\right)^{k(n)}\\
\le \frac{(1-p)n}{p}\int_0^{1+1/n}\zeta^{k(n)}\,d\zeta\\
=\frac{(1-p)n}{p}\cdot \frac{(1+1/n)^{k(n)+1}}{k(n)+1}
=O\left( \frac{(1-p)n}{pcn^{1/2}}\right),
\end{multline*}
if $c=O(n^{1/2})$. Besides, the expected number of unmarked records located somewhere in $[k(n)]$ is at most $(1-p)k(n)$. Thus, combining the last two bounds and \eqref{0.1}, we conclude that $\Bbb E[\la_{u}(X)]$ 
is at most of order 
\[
(1-p)\sqrt{\frac{n}{p}}+c(1-p)n^{1/2}+\frac{n^{1/2}(1-p)}{pc}=O\left(\!(1-p)\sqrt{\frac{n}{p}}\right),
\] 
if we choose $c=p^{-1/2}$, which we can do since $p\ge n^{-1}$.
\end{proof}

\begin{theorem}\label{thm2} For every $p\in [0,1]$,
\begin{equation*}
\max_{\bar X}\Bbb E[\la_{\text{m}}(X)]\le \sum_{j\in [n]} \frac{1-(1-p)^j}{j},
\end{equation*}
which is the equality if and only if $\bar X=(1,2,\dots,n)$.
\end{theorem}

\begin{proof} Let $B_{\ell}$ be the event that a generic element $\ell\in [n]$ is a marked left-to-right maximum of
$X$. Then $B_{\ell}=\cup_{k\le\ell}B_{\ell,k}$; here $B_{\ell,k}$ is the event ``$\ell =x_k$ and $x_k$ is a marked left-to-right maximum of $X$''. Let us bound $\Bbb P(B_{\ell,k})$.

Introduce the sets $N_{i,j}$:
\begin{equation*}
\begin{aligned}
&N_{1,1}=\{\bar x_{\a}:\a<k,\,  \bar x_{\a}<\ell \}, \,\,\, N_{1,2}=\{\bar x_{\a}:  \a<k,\, \bar x_{\a}>\ell\},\\
&N_{2,1}=\{\bar x_{\a}:\a>k,\,  \bar x_{\a}<\ell \}, \,\,\, N_{2,2}=\{\bar x_{\a}:  \a>k,\, \bar x_{\a}>\ell\},\\
\end{aligned}
\end{equation*}
and $n_{i,j}=|N_{i,j}|$. Needless to say, $N_{i,j}$, $n_{i,j}$ depend on $\ell$ and $k$, and $\sum_{i,j}n_{i,j}=n-1$ on $B_{\ell,k}$.  We will use
\begin{equation}\label{use}
n_{\circ,1}:=n_{1,1}+n_{2,1}=\ell-1-\Bbb I(\bar x_k<\ell)\le \ell-1.
\end{equation}
Let
$M_{i,j}$ stand for the set of marked elements of $N_{i,j}$, and $m_{i,j}:=|M_{i,j}|$.  On $B_{\ell,k}$, once the marked elements different from $\ell$ are permuted, the marked element $\ell$ takes the position immediately following $(k-1)$-th element of that $(n-1)$-long sub-permutation. It means that necessarily $m_{1,2}=n_{1,2}$: otherwise an unmarked element $\bar x_{\a}\in N_{1,2}\setminus M_{1,2}$ (with $\a<k$) precedes $\ell$, preventing $\ell$ from being a left-to-right maximum in $X$.

Let $k_1$ and $k_2$ denote generic numbers of elements from $M_{1,1}$ and from $M_{2,1}$ that will stay among the first $(k-1)$ elements (among the last $n-k$ elements) of a resulting admissible permutation of $\bar X\setminus\{\ell\}$. For the event in question to hold, $k_1+k_2=k-1$. The total number of such permutations of $\bar X\setminus \{\ell\}$ with given $\{m_{i,j}\}$ is
\begin{multline*}
(k-1)!\bigl(m-(k-1)\bigr)!\prod_{(i,j)\neq (1,2)}\binom{n_{i,j}}{m_{i,j}}\cdot \sum_{k_1+k_2=k-1}\binom{m_{1,1}}{k_1}\binom{m_{2,1}}{k_2}\\
=(k-1)!\bigl(m-(k-1)\bigr)!\prod_{(i,j)\neq (1,2)}\binom{n_{i,j}}{m_{i,j}}\cdot\binom{m_{1,1}+m_{2,1}}{k-1};
\end{multline*}
$(1,2)$ is excluded from the double product since $m_{1,2}=n_{1,2}$. For every such permutation we get an admissible permutation of $\bar X$ by inserting element $\ell$ right after $(k-1)$-th element of that permutation. 
For this to happen, if $\bar x_k\neq \ell$, $\bar x_k$ needs to be marked and moved out of the slot $k$. Neglecting
this restraint, we bound the contribution of all such permutations with a given $m$ to $\Bbb P(B_{\ell,k})$,
by multiplying the above product by 
$\tfrac{p^{m+1}(1-p)^{n-m-1}}{(m+1)!}$ ($\ell$ must be marked too!) and evaluating the sum of all these fractions over $\{m_{i,j}\}$ subject to
constraints $\sum_{i,j}m_{i,j}=m$, $m_{1,2}=n_{1,2}$. So, within the factor $\tfrac{p^{m+1}(1-p)^{n-m-1}}{(m+1)!}\cdot (k-1)!(m-(k-1))!$, 
the contribution in question is at most 
\begin{equation*}
\sum_{m_{i,j}\le n_{i,j}\atop \sum_{(i,j)\neq (1,2)}m_{i,j}=m-n_{1,2} }
\!\!\!\prod_{(i,j)\neq (1,2)}\binom{n_{i,j}}{m_{i,j}}\,\,\times  \binom{m_{\circ,1}}{k-1},\\
\end{equation*}
where $m_{\circ,1}:=m_{1,1}+m_{2,1}$. Given $m$ and $m_{\circ,1}\in [k-1,m-n_{1,2}]$, we have $m_{2,2}=m-n_{1,2}-m_{\circ, 1}$, and the sum over $m_{1,1}$ and $m_{2,1}$ equals
\begin{multline*}
\binom{m_{\circ,1}}{k-1}\sum_{m_{1,1}+m_{2,1}=m_{\circ,1},\atop m_{2,2}=m-n_{1,2}-m_{\circ,1}}
\binom{n_{1,1}}{m_{1,1}}\binom{n_{2,1}}{m_{2,1}}\binom{n_{2,2}}{m_{2,2}}\\
=\binom{m_{\circ,1}}{k-1}\binom{n_{\circ,1}}{m_{\circ,1}}\binom{n_{2,2}}{m-n_{1,2}-m_{\circ,1}}\\
=\binom{n_{\circ,1}}{k-1}\binom{n_{\circ,1}-(k-1)}{m_{\circ,1}-(k-1)}\binom{n_{2,2}}{m-n_{1,2}-m_{\circ,1}},
\end{multline*}
where, by \eqref{use},  $n_{\circ,1}:=n_{1,1}+n_{2,1}\le\ell-1$. So, given $m$, the sum over all $m_{\circ,1}\in [k-1,m-n_{1,2}]$ is bounded by 
\[
\binom{\ell-1}{k-1}\binom{n-1-n_{1,2}-(k-1)}{m-n_{1,2}-(k-1)},
\]
which is zero, unless $\ell\ge k$. So, the counterpart of the formula \eqref{new1} is 
\begin{multline}\label{new2}
\Bbb P(B_{\ell,k})=\sum_{m\ge n_{1,2}+k-1} \Bigl[(m+1)\binom{m}{k-1}\Bigr]^{-1}\\
\times \binom{\ell-1}{k-1}\binom{n-1-n_{1,2}-(k-1)}{m-n_{1,2}-(k-1)}\cdot p^{m+1}(1-p)^{n-1-m}.
\end{multline}
Using \eqref{classic},  we transform \eqref{new2} into 
\begin{multline*}
\Bbb P(B_{\ell,k})
=(1-p)^{n-1}p\,\int_0^1\binom{\ell-1}{k-1}\left(\frac{px}{1-p}\right)^{k-1}
\left(\frac{p(1-x)}{1-p}\right)^{n_{1,2}}\\
\times\sum_{m\ge k-1+n_{1,2}}\binom{n-k-n_{1,2}}{m-k-n_{1,2}+1}\left(\frac{p(1-x)}{1-p}\right)^{m-k-n_{1,2}+1}\,dx\\
=(1-p)^{n-1}p\int_0^1\left(\frac{p(1-x)}{1-px}\right)^{n_{1,2}}\binom{\ell-1}{k-1}\left(\frac{px}{1-p}\right)^{k-1}\left(\frac{1-px}{1-p}\right)^{n-k}\,dx\\
\le(1-p)^{n-1}p\int_0^1
\binom{\ell-1}{k-1}\left(\frac{px}{1-p}\right)^{k-1}\left(\frac{1-px}{1-p}\right)^{n-k}\,dx,
\end{multline*}
since $\tfrac{p(1-x)}{1-px}\le 1$. Crucially, the bottom RHS does not depend on $n_{1.2}$. And 
\begin{multline*}
\sum_{k\in [\ell]} \binom{\ell-1}{k-1}\left(\frac{px}{1-p}\right)^{k-1}\left(\frac{1-px}{1-p}\right)^{n-k}\\
=\left(\frac{1-px}{1-p}\right)^{n-\ell}\left(\frac{px}{1-p}+\frac{1-px}{1-p}\right)^{\ell-1}=\frac{(1-px)^{n-\ell}}{(1-p)^{n-1}}.\\
\end{multline*}
So, summing $\Bbb P(B_{\ell,k})$ over $k\in [\ell]$, we conclude that 
\begin{equation*}
\Bbb P(B_{\ell})\le p\int_0^1(1-px)^{n-\ell}\,dx=\tfrac{1-(1-p)^{n-\ell+1}}{n-\ell+1}.
\end{equation*}
Therefore $\Bbb E[\la_{\text{m}}]\le \sum_{j\in [n]}\tfrac{1-(1-p)^j}{j}$, and this upper bound is attained if and only if, for each $\ell$ and $k$, we have $n_{1,2}=n_{1,2}(\ell,k)=0$, which happens if and only if
$\bar X=(1,2,\dots,n)$.
\end{proof}
\begin{corollary}\label{cor1} For $p\gg n^{-1}, 1-p\ge \text{const }n^{-1/2}\log n$,  $\max_{\bar X}\Bbb E[\la(X)]=O\Bigl((1-p)\sqrt{\tfrac{n}{p}}\Bigr)$.
\end{corollary}

\begin{proof} In view of Theorems $1,2$, we need to prove that for $p$ in question
\[
\sum_{j\in [n]}\frac{1-(1-p)^j}{j}=O\left(\!(1-p)\sqrt{\frac{n}{p}}\right).
\]
Denote $H_n=\sum_{j\in [n]}\tfrac{1}{j}$. Since $1/j$ decreases with $j$, we have $H_n\le 1+\int_1^n\tfrac{dx}{x}=\log(ne)$. So,
\begin{multline*}
\sum_{j\in [n]}\frac{1-(1-p)^j}{j}=H_n-\sum_{j\ge 1}\frac{(1-p)^j}{j}+\sum_{j>n}\frac{(1-p)^j}{j}\\
\le \log(ne)+\log p+\frac{1}{(n+1)p}
\le\log(npe^2)=O\left(\!(1-p)\sqrt{\frac{n}{p}}\right),
\end{multline*}
which is true since $1-p\ge \text{const }n^{-1/2}\log n$.
\end{proof}

Banderier, Beier, and Mehlhorn \cite{BMB} discovered that for a permutation
\[
\bar X=(n-k,n-k+1,\dots,n, 1,2,\dots,n-k-1),\quad k:=\big\lfloor\sqrt{n/p}\big\rfloor,
\]
$\Bbb E[\la(X)]$ is {\it at least\/} of order $\sqrt{n/p}$, if $p<1/2$, and fixed. And then they obtained an almost matching upper bound $O(\sqrt{(n/p)\log n})$ that holds uniformly for every $\bar X$.

Slightly modified, the ingenious argument in \cite{BMB} has enabled us to prove that, for the same $\bar X$,  $\Bbb E[\la(X)]$ is {\it at least\/} of order $(1-p)\sqrt{n/p}$, if $p\gg n^{-1}$ and $1-p\gg n^{-1/2}$.
Combining this stronger result and the corollary we obtain that $\max_{\bar X}\Bbb E[\la(X)]$ is of order 
$(1-p)\sqrt{n/p}$ exactly if $p\gg n^{-1}$ and $1-p\ge \text{const }n^{-1/2}\log n$. 
\bigskip

{\bf Acknowledgment.\/} Years ago, my brother-in-law Roma Goldenberg mailed me from Russia three first volumes of Don Knuth's ``The Art of Computer Science'', just translated into Russian. Reading these books turned out to be a life changing experience for me. Thanks to Don, I learned about the work of Cyril Banderier, Rene Beier, and Kurt Mehlhorn during the editing phase for the fourth volume in that series.

\end{document}